\documentclass[12pt]{huber_article}

\usepackage{amsthm,amsfonts,amsmath,dsfont}

\usepackage{comment}

\newcommand{\prob}{\mathbb{P}}
\newcommand{\mean}{\mathbb{E}}
\newcommand{\var}{\mathbb{V}}
\newcommand{\mgf}{\operatorname{mgf}}
\newcommand{\ind}{\mathds{1}}
\newcommand{\unifdist}{\textsf{Unif}}

\newtheorem{lemma}{Lemma}

\usepackage{tikz}

\usepackage[utf8]{inputenc}
\begin{document}

\title{Halving the bounds for the Markov, \\
Chebyshev, and 
Chernoff Inequalities using smoothing}

\author{\noindent Mark Huber\\ Claremont McKenna College \\ {\tt mhuber@cmc.edu}}

\date{}

\maketitle

\begin{abstract}
  The Markov, Chebyshev, and Chernoff inequalities are
  some of the most widely used methods for bounding
  the tail probabilities of random variables.  In all three
  cases, the bounds are tight in the sense that there
  exists easy examples where the inequalities become
  equality.  Here we will show that through a simple
  smoothing using auxiliary randomness, that each of 
  the three bounds can be cut in half.  In many common
  cases, the halving can be achieved without the need
  for the auxiliary randomness.
\end{abstract}

\section{Introduction}

Markov's inequality, Chebyshev's inequality, and Chernoff's
inequality are three of the most widely used equalities
in applied probability.  Chernoff's 1952~\cite{chernoff1952}
paper alone has over 3500 citations, and the Markov 
and Chebyshev inequalities appear in virtually every 
undergraduate probability textbook.

\paragraph{Markov's inequality}
This inequality (see for instance~\cite{ross2006}) applies to all nonnegative random variables with finite mean.  It 
can be written as
\begin{equation}
\label{EQN:markov}
(\forall a \geq 0)(\prob(X \geq a) \leq \mean[X]/a).
\end{equation}

This inequality is tight.  Consider the simple random
variable that places all of its probability mass
at either $a$ or $0$. Then 
\[
\mean[X] = a\prob(X = a) = a \prob(X \geq a),
\]
so equality is obtained.

\paragraph{Chebyshev's inequality} The next inequality (see for instance~\cite{ross2006}) assumes both a finite first and second moment,
and so the variance $\var(X)$ is finite.  The bound is then 
\begin{equation}
\label{EQN:chebyshev}
(\forall a \geq 0)(\prob(|X - \mean[X]| \geq a) \leq 
 \var(X)/a^2).
\end{equation}

This bound is also tight.  Consider $X$ where 
$\prob(X = a) = \prob(X = -a) = p/2$, and
$\prob(X = 0) = 1 - p$.  Then $\mean[X] = 0$, 
$\var(X) = pa^2$, and 
\[
\prob(|X - \mean[X]| \geq a) = p = \var(X)/a^2.
\]

\paragraph{Chernoff's bound}  The Chernoff 
bound~\cite{chernoff1952} technically
applies to all random variables $X$, but is most effective
when there exist $t > 0$ and $t < 0$ such that the 
moment generating function $\mgf_X(t) = \mean[\exp(tX)]$ is finite.
It consists
of a bound on the right tail 
\begin{equation}
\label{EQN:upperchernoff}
(\forall a)(\forall t \geq 0)(\prob(X \geq a) \leq  \mgf_X(t)\exp(-ta)),
\end{equation}
and a bound on the left tail
\begin{equation}
\label{EQN:lowerchernoff}
(\forall a)(\forall t \leq 0)(\prob(X \leq a) \leq \mgf_X(t)\exp(-ta)),
\end{equation}

As with the Markov and Chebyshev inequalities, these upper and 
lower bounds are tight.  For the upper bound with $a \geq 0$, let
$\prob(X = a) = p_a$ and $\prob(X = 0) = 1 - p_a$.
Then
\[
\mgf_X(t)\exp(-ta) = [p_a e^{ta} + (1 - p_a)\exp(0)]\exp(-ta)
 = p_a + (1 - p_a)\exp(-ta).
\]
Since as $t \rightarrow \infty$, this gives an upper bound 
arbitrarily close to $p_a$, this is tight.  The lower tail bound
has a similar tight example.

\paragraph{Sample averages}

An important use of these tail inequalities is when the random
variable $X$ is the sample average of $n$ independent identically
distributed random variables $Y_1,\ldots,Y_n$.  That is,
$X = (Y_1 + \cdots + Y_n)/n$.  This is an important ingredient
in Monte Carlo simulation.  If the $Y_i$ have mean $\mu$,
then so does $X$, and the sample average can be 
used as an estimate of $\mu$.
Tail bounds then can be used to show how unlikely it is that
the estimate is far away from the mean.

Since Markov's inequality only depends upon the value of 
$\mean[X]$, increasing $n$ does not improve the bound.
However, if $Y_i$ has finite standard deviation $\sigma$,
then the standard deviation of $X$ is $\sigma/\sqrt{n}$.
Then Chebyshev's inequality can be used to say that 
$\prob(|X - \mu| \geq a) \leq \sigma^2/[n a^2],$ and
so the chance that $X$ is far away from its mean is inversely proportional to $n$.

To do better than polynomial convergence, methods such as the {\em median-of-means} approach to estimating 
$\mu$ (see~\cite{jerrumvv1986}) are used.
Suppose there exists a value of $t > 0$ where
$\mgf_{Y_i}(t)$ is finite, then 
\begin{align*}
 \prob(X \geq a) &= \prob(Y_1 + \cdots + Y_n \geq n a) \\
 &\leq \mgf_{Y_1 + \cdots + Y_n}(t) \exp(-t na).
\end{align*}
For independent random variables, the moment generating function
of the sum is the product of the moment generating function, so 
\[
\prob(X \geq a) \leq \mgf_{Y_1}(t)^n \exp(-tna) = 
 [\mgf_{Y_1}(t) \exp(-ta)]^n.
\]

Hence Chernoff bounds show that the probability that
$X$ is far away from its mean decreases exponentially
in $n$.  (The lower tail analysis is similar.)

A Chernoff type bound is central to the 
$M$-estimator for $\mu$
of Catoni~\cite{catoni2012}, and approximation
algorithms derived from it~\cite{huberPrePrint14}.
Any improvement in the Chernoff bound through
smoothing leads immediately to an improvement in 
the error bounds of these algorithms.

\section{Smoothing the Markov inequality}

In order to improve these inequalities using
smoothing, we need 
two simple facts about expected value.

\begin{lemma}
For measurable functions $f$ and $g$ with 
$f(x) \leq g(x)$ for all $x$, and a random variable 
$X$ such that $f(X)$ and $g(X)$ are integrable,
\[
\mean[f(X)] \leq \mean[g(X)].
\]
\end{lemma}

\begin{lemma}
Let $\ind(\cdot)$ denote the indicator function that is
1 when the argument is true and 0 when it is false.
Then
\[
\mean[\ind(X \in A)] = \prob(X \in A).
\]
\end{lemma}

Combined, this gives a simple proof of Markov's inequality.
\begin{lemma}
For any $a \geq 0$ and integrable random variable $X$,
$\prob(X \geq a) \leq \mean[X]/a.$
\end{lemma}

\begin{proof}
Note that $\ind(x \geq a) \leq (x/a)\ind(x \geq 0)$ (see Figure~\ref{FIG:markovbound}.)  So
\[
\prob(X \geq a) = \mean[\ind(X \geq a)]
 \leq \mean[X/a] = \mean[X]/a.
\]
\end{proof}

\begin{center}
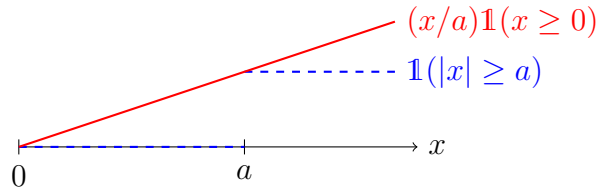
\begin{figure}[ht]
\begin{center}
 \begin{tikzpicture}
 \draw[->] (0,0) -- (5.3,0) node[right] {$x$};
 \draw (3,0.1) -- (3,-0.1) node[below] {$a$};
 \draw (0,0.1) -- (0,-0.1) node[below] {$0$};
 \draw[thick,blue,dashed] (0,0) -- (3,0);
 \draw[thick,blue,dashed] (3,1) -- (5,1) node[right] {$\ind(|x| \geq a)$};
 \draw[thick,red] (0,0) -- (5,5/3) node[right]
   {$(x/a)\ind(x \geq 0)$};
 \end{tikzpicture}
 \end{center}
 \caption{Bounding function to show Markov's inequality}
 \label{FIG:markovbound}
 \end{figure}	
\end{center}

Now consider the tight example from earlier where all of the probability
mass is either at 0 or at $a$.  But then suppose we add
a uniform random variable centered at $0$ to $X$.  
Write $U \sim \unifdist([-c,c])$.  Then if $X$ is either
at 0 or $a$ (and $c \leq a$), then there is a $1/2$ chance
that $X + U < a$ so the bound is halved for the original
tight example.

Adding a random variable with mean 0 to $X$ does
not change the mean, which is important for the 
Monte Carlo applications mentioned earlier.

Of course, if $c$ is small, then the formerly tight
example can be 
altered slightly by moving the probability mass at $a$ to 
$a + c$, in which
case $X + U \geq a$ if $X = a$.  As $c$ increases, 
however, this becomes harder to do while keeping 
$p_a$ large.

To deal with this and other possibilities, it helps
to note that
\[
\prob(X + U \geq a) = \mean[\prob(X + U \geq a|X)].
\]

Let $f_1(X) = \prob(X + Y \geq a|X)$.  Then
$f_1$ is piecewise linear, and 
connects the points (0,0), $(a-c,0)$, $(a+c,1)$ and
then is constant 1 for values beyond $(a+c)$ (see
the dotted line in Figure~\ref{FIG:markovsmooth}.)
\begin{lemma}
\label{LEM:markovsmooth}
For $U \sim \unifdist([-c,c])$ where $c \leq a$ independent
of a nonnegative random variable $X$ with finite mean,
$\prob(X + U \geq a) = \mean[f_1(X)]$ where 
\[
f_1(x) = \ind(x \geq a + c) +
 \frac{x-a+c}{2c} \ind(x \in [a-c,a+c]).
\]
\end{lemma}

\begin{proof}
  Start with
  \begin{align*}
  \prob(X + U \geq a) &= \mean(\ind(X + U \geq a)) = \mean[\mean[\ind(X + U \geq a)|X]] \\
   &= \mean[\mean[\ind(U \geq a - X)|X]] 
   = \mean[\prob(U \geq a - X|X)].
   \end{align*}
   If $X \geq a + c$, then $a - X \leq -c$
   and $\prob(U \geq a - X|X) = 1$.  Similarly, if 
   $X \leq a - c$, then $a - X \geq c$ and 
   $\prob(U \geq a - X|X) = 0$.
   
   If $X \in [a-c,a+c]$, then 
   \[
   \prob(U \geq a - X|X) = (c-(a-X))/(2c) 
     = (X-a+c)/(2c).
     \]	
     which gives the result.
\end{proof}

\begin{center}
\begin{figure}[ht]
\begin{center}
 \begin{tikzpicture}
 \draw[->] (0,0) -- (6.3,0) node[right] {$x$};
 \draw (3,0.1) -- (3,-0.1) node[below]  {$a$};
 \draw (4,0.1) -- (4,-0.1) node[below]  {$a+c$};
 \draw (2,0.1) -- (2,-0.1) node[below]  {$a-c$};
 \draw (0,0.1) -- (0,-0.1) node[below]  {$0$};
 \draw[thick,blue,dashed] (0,0) -- (2,0) --
   (4,1) -- (6,1) node[right] {$f_1(x)$};
 \draw[thick,red] (0,0) -- (6,6/4) node[right]
   {$x/(a+c)$};
 \end{tikzpicture}
 \end{center}
 \caption{Additive smoothing for $X$}
 \label{FIG:markovsmooth}
 \end{figure}
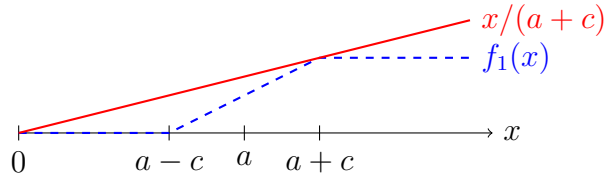	
\end{center}

To make the bounding line $x/(a+c)$ as small as possible,
we  
should set $c$ to be as large as possible.
Making $c = a$ gives a Markov inequality with 
a bound that is one half of what it was originally.

\begin{lemma}
  For integrable $X$, $a \geq 0$, and $U$ a random
  variable independent of $X$ that is uniform 
  over $[-a,a]$,
  \[
  \prob(X + U \geq a) \leq (1/2)\mean[X]/a
  \]
\end{lemma}

Note that in the particular case that $X$ is a continuous 
random variable with decreasing density, then it is not 
necessary to add the smoothing variable $U$ to achieve
this improved bound (see Figure~\ref{FIG:density}.)  

\begin{lemma}
\label{LEM:removeu}
  Let $X$ be a nonnegative continuous random variable with
  decreasing density, $0 \leq c \leq a$, and $U$ 
  be a random variable
  independent of $X$ such that $U \sim \unifdist([-c,c])$, then 
  $\prob(X + U \geq a) \geq \prob(X \geq a)$.
\end{lemma}

\begin{proof}
  First let us consider the probability we are looking for.
  \begin{align*}
  \prob(X + U \geq a) &= \int_{-c}^c \frac{1}{2c} \prob(X \geq a - u) \ du
  \end{align*}
  Consider $u \in [-c,c] = [-c,0] \cup [0,c]$.  
  When $u \geq 0$, $[a-u,\infty) = [a,\infty) \cup [a-u,a]$.  When $u \leq 0$, we have
  $[a-u,\infty) = [a,\infty) \setminus [a,a-u]$.
  So
  \begin{align*}
  \prob(X + U \geq a)
  &= \int_{-c}^{c} \frac{\prob(X \geq a)}{2c} \ du 
     + 
     \int_0^c 
     \frac{\prob(X \in [a-u,a]}{2c} \ du
     - \int_{-c}^0 
     \frac{\prob(X \in [a,a-u])}{2c} \ du 
  \end{align*}
  By using substitution in the last integral 
  to change the sign we obtain
  \begin{align*}
  \prob(X + U \geq a)
  &= \prob(X \geq a) 
     + \frac{1}{2c}\int_{0}^c 
     \prob(X \in [a-u,a]) -
     \prob(X \in [a,a+u]) \ du.
  \end{align*}
  Because of the declining 
  density, the second integral
  is nonnegative which gives the result.
\end{proof}

\begin{center}
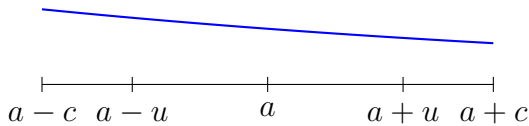
\begin{figure}[ht]
\begin{center}
 \begin{tikzpicture}
 \draw (0,0) -- (6,0);
 \draw (3,0.1) -- (3,-0.1) node[below] {$a$};
 \draw (6,0.1) -- (6,-0.1) node[below] {$a+c$};
 \draw (0,0.1) -- (0,-0.1) node[below] {$a-c$};
 \draw (4.8,0.1) -- (4.8,-0.1) node[below] {$a+u$};
 \draw (1.2,0.1) -- (1.2,-0.1) node[below] {$a-u$};
 \draw[thick,blue] plot[domain=0:6] (\x,{exp(-0.1*\x)});
\end{tikzpicture}
 \end{center}
 \caption{The probability $X$ is near $a - u$
 and $X+U \geq a$ (so $U$ is near $u$) is greater than the 
 chance that $X$ is near $a + u$ and $U$ is near $-u$ so that 
 $X + U < a$.}
 \label{FIG:density}
 \end{figure}	
\end{center}

For example, say $X$ is an exponential
random variable with density
$f_X(x) = \exp(-x)$ for $x \geq 0$
(and 0 otherwise.)  Then $\mean[X] = 1$, so $\prob(X \geq 1) \leq 1/2$ with Lemma~\ref{LEM:markovsmooth}
whereas the regular Markov inequality gives an upper bound of 1.
The exact tail probability is
$\exp(-1) = 0.3678\ldots$.

\section{Smoothing the Chebyshev inequality}

To show Markov's inequality
we used a bounding
line, for Chebyshev we use a bounding parabola.
See Figure~\ref{FIG:chebyshevbound}.
\begin{lemma}
  For $a \geq 0$, $\ind(|x - \mu| \geq a) \leq (x-\mu)^2/a^2$.  \end{lemma}
The proof is straightforward, and the result
immediately gives Chebyshev's inequality.
\begin{lemma}
  For a random variable $X$ with finite first and second
  moments, $\prob(|X - \mu| \geq a) \leq \var(X)/a^2$.
\end{lemma}

\begin{proof}
\[
\prob(|X - \mu| \geq a) =\mean[\ind(|X - \mu| \geq a]
 \leq \mean\left[\frac{(X - \mu)^2}{a^2}\right] = \frac{\var(X)}{a^2}.
\]
\end{proof}

\begin{center}
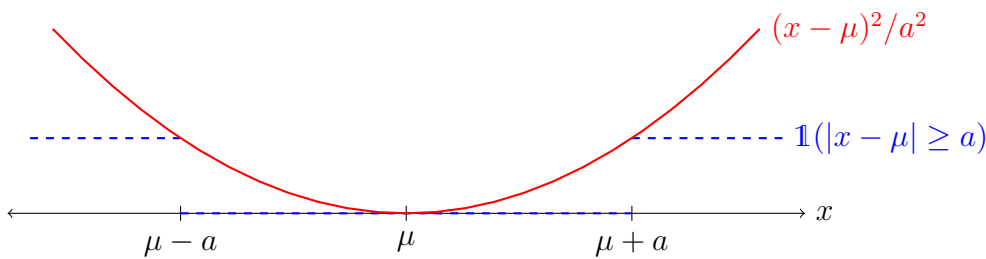
\begin{figure}[ht]
 \begin{tikzpicture}
 \draw[<->] (-5.3,0) -- (5.3,0) node[right] {$x$};
 \draw (3,0.1) -- (3,-0.1) node[below] {$\mu + a$};
 \draw (-3,0.1) -- (-3,-0.1) node[below] {$\mu -a$};
 \draw (0,0.1) -- (0,-0.1) node[below] {$\mu$};
 \draw[thick,blue,dashed] (-5,1) -- (-3,1);
 \draw[thick,blue,dashed] (-3,0) -- (3,0);
 \draw[thick,blue,dashed] (3,1) -- (5,1) node[right] 
   {$\ind(|x - \mu| \geq a)$};
 \draw[thick,red] plot[domain=-4.7:4.7] (\x,{\x*\x/9}) node[right]
   {$(x-\mu)^2/a^2$};
 \end{tikzpicture}
 \caption{Bounding function to show Chebyshev's inequality.}
 \label{FIG:chebyshevbound}
 \end{figure}	
\end{center}

For Chebyshev, we wish to once again smooth the 
random variable as much as possible by adding
$U \sim \unifdist([-c,c])$ to $X$.  Note that it is 
not possible to reduce the bounding function by more
than a factor of $2$, since at $a$ the smoothed
function will be linearly interpolating between the 
value 0 at $a - c$ and 1 at $a + c$.  Write
$\prob(|X + U - \mu| \geq a|X) = f_2(X)$.

Then we must choose a value for $c$ so that 
$f_2(x) \leq (1/2)(x - \mu)^2/a^2$.  The easiest
way to do this is to make $f_2(x)$ tangent to the 
parabola at $\mu+a$ and $\mu - a$.  This happens
when $c = a/2$. See
Figure~\ref{FIG:chebyshevsmooth}.

\begin{lemma}
  Let $X$ have finite first and second moments, 
  $a \geq 0$, and 
  $U \sim \unifdist([-(1/2)a,(1/2)a]$ be independent
  of $X$.  Then
  \[
  \prob(|X + U - \mu| \geq a) \leq \frac{1}{2} \cdot 
    \frac{\var(X)}{a^2}.
  \]
\end{lemma}

\begin{proof}
  As in the proof of Lemma~\ref{LEM:markovsmooth},
  it is straightforward to show that 
  $\prob(|X + U - \mu| \geq a) = \mean[f_2(X)],$
  where 
  \[
  f_2(x) = \ind(|x| > a) + \frac{|x|-(1/2)a}{a}\ind(|x| \in [(1/2)a,(3/2)a]
  \]
\end{proof}

\begin{center}
\begin{figure}[ht]
 \begin{tikzpicture}[yscale=1.3]
 \draw (-5.3,0) -- (5.3,0) node[right] {$x$};
 \draw (3,0.1) -- (3,-0.1) node[below] {$\mu + a$};
 \draw (-3,0.1) -- (-3,-0.1) node[below] {$\mu - a$};
 \draw (4.5,0.1) -- (4.5,-0.1) node[below] {$\mu + \frac{3}{2}a$};
 \draw (-4.5,0.1) -- (-4.5,-0.1) node[below] {$\mu - \frac{3}{2}a$};
 \draw (1.5,0.1) -- (1.5,-0.1) node[below] {$\mu + \frac{1}{2}a$};
 \draw (-1.5,0.1) -- (-1.5,-0.1) node[below] {$\mu - \frac{1}{2}a$};
 \draw (0,0.1) -- (0,-0.1) node[below] {$\mu$};
 \draw[thick,blue,dashed] (-5,1) -- (-4.5,1) -- (-1.5,0) --
   (1.5,0) -- (4.5,1) -- (5,1) node[right] {$f_2(x)$};
 \draw[thick,red] plot[domain=-5:5] (\x,{\x*\x/18}) node[right]
   {$(1/2)(x-\mu)^2/a^2$};
 \end{tikzpicture}
 \caption{Smoothed function for Chebyshev.}
 \label{FIG:chebyshevsmooth}
 \end{figure}
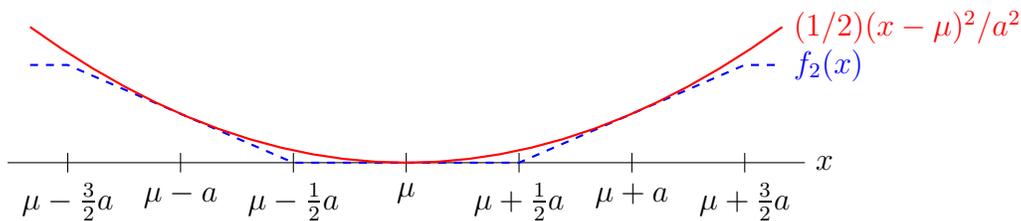	
\end{center}

Using Lemma~\ref{LEM:removeu} on
$X$ and $-X$, it is possible to
show that adding $U$ is unnecessary
for certain random variables.
\begin{lemma}
\label{LEM:smooth}  Let $a\geq 0$.
  Let $X$ be a continuous 
  nonnegative random variable
  with decreasing density over the interval 
  $[(1/2)a,(3/2)a]$ and increasing
  density over $[-(3/2)a,-(1/2)a]$.
  Then 
  \[
  \prob(|X - \mu| \geq a) \leq 
  \frac{1}{2} \frac{\var(X)}{a^2}.
  \]
\end{lemma}

For instance, for a standard normal
random variable $Z$, 
this upper bounds the probability  
that $|Z| \geq 1$ by 1/2 (the true
probability $|Z| \geq 1$ is about 0.3173.)

This result 
is similar to a classic
result of Gauss~\cite{gauss1823} (presented
in the next lemma), although that
result only applies for $a \geq (4/3)\mean[X^2]$.

\begin{lemma}
  For $X$ a continuous random variable
  that is unimodal with mode 0 and
  $a^2 \geq (4/3)\mean[X^2]$,
  \[
  \prob(|X| \geq a) \leq 
    \frac{4}{9} \cdot \frac{\mean[X^2]}{a^2}
  \]
\end{lemma}

\section{Smoothing the
Chernoff inequality}

Now consider the bound on the upper tail for Chernoff
where we are trying to bound 
$\prob(X \geq a) = \mean[\ind(X \geq a)].$  Here
we use the fact that $\ind(X \geq a) \leq \exp(t(X - a))$.
(see Figure~\ref{FIG:chernoffbound}.)  This immediately
gives the upper Chernoff bound.

\begin{lemma}
 For any random variable $X$ and $t$ such that 
 $\mgf_X(t) = \mean[\exp(tX)]$ is finite,
 \[
 \prob(X \geq a) \leq \mean[\exp(t(X - a))].
 \]
\end{lemma}

\begin{center}
\begin{figure}[ht]
 \begin{tikzpicture}[xscale=0.8,yscale=1.2]
 \begin{scope}[xshift=2.6in]
 \draw (0,0) -- (5,0);
 \draw (3,0.1) -- (3,-0.1) node[below] {$a$};
 \draw (5,0.1) -- (5,-0.1) node[below] {$a+1/t$};
 \draw (1,0.1) -- (1,-0.1) node[below] {$a-1/t$};
 \draw[thick,blue,dashed] (0,0) -- (1,0) -- (5,1) -- (6,1)       node[right] {$f_3(x)$};
 \draw[thick,red] plot[domain=0:5] (\x,{exp(0.5*(\x-3))/2}) node[right]
   {$(1/2)\exp(t(x - a))$};
   \end{scope}
 \begin{scope}
 \draw (0,0) -- (6,0);
 \draw (3,0.1) -- (3,-0.1) node[below] {$a$};
 \draw[thick,blue,dashed] (0,0) -- (3,0);
 \draw[thick,blue,dashed] (3,1) -- (5,1) node[right] 
   {$\ind(x \geq a)$};
 \draw[thick,red] plot[domain=0:5] (\x,{exp(0.5*(\x-3))}) node[right]
   {$\exp(t(x - a))$};
   \end{scope}
 \end{tikzpicture}
 \caption{Bounding function and smoothed function for Chernoff's inequality.  Here $\mean[\ind(x \geq a)] = \prob(X \geq a)$ and $\mean[f_3(X)] = \prob(X + U \geq a)$.}
 \label{FIG:chernoffbound}
 \end{figure}
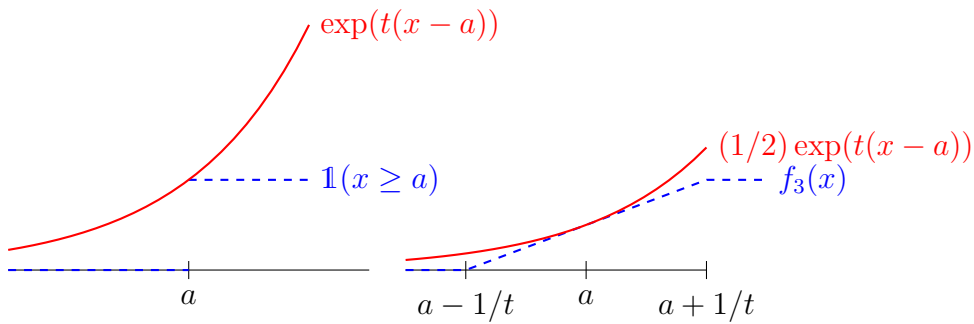	
\end{center}

When we consider $(1/2)\exp(t(x-a))$, this has derivative
of $(1/2)t$ at $x = a$.  So make the width of the interval
equal to $2/t$.  Doing so gives the smoothed bound for Chernoff.
\begin{lemma}
  Let $X$ be a random variable such that for 
  $t \geq 0$, $\mgf_X(t)$ is finite, and $U$
  be a uniform random variable over $[-1/t,1/t]$ that
  is independent of $X$.  Then
  \[
  \prob(X + U \geq a) \leq (1/2)\mgf_X(t)\exp(-ta).
  \]
  For $t \leq 0$ such that $\mgf_X(t)$ is finite, then
  for $U \sim \unifdist([1/t,-1/t])$ a random variable that
  is independent of $X$,
\[
  \prob(X + U \leq a) \leq (1/2)\mgf_X(t)\exp(-ta).  
  \]
\end{lemma}

Again using Lemma~\ref{LEM:removeu}, the upper
tail bound applies to many variables without
using the smoothing.
\begin{lemma}
\label{LEM:chernoffremoveu}
  Let $a \geq 0$ and $t \geq 0$.  
  Suppose $X$ has finite moment generating
  function at $t$, and that $X$ 
  is a continuous random variable
  with a density that is decreasing over
  $[a-1/t,a+1/t]$ then 
  \[
  \prob(X \geq a) \leq (1/2)\mgf_X(t)\exp(-ta).
  \]
\end{lemma}

Again consider a standard random variable
$Z$.  For $a = 1$, $t = 1$, the density
of $Z$ is decreasing over $[0,2]$, so
the original Chernoff bound of 
$\prob(Z \geq 1) \leq \exp(-1/2)$ can
be reduced using Lemma~\ref{LEM:chernoffremoveu} to
$\prob(Z \geq 1) \leq (1/2)\exp(-1/2) \approx 0.3012$ which is much closer to the exact
tail probability of about 0.1586.

\end{document}